\documentclass{amsart}
\usepackage{amsmath}%
\usepackage{amsfonts}%
\usepackage{amssymb}%
\usepackage{graphicx}
\usepackage{mathrsfs}
\usepackage{amsxtra}
\usepackage{subfigure}
\usepackage[colorlinks]{hyperref}

\theoremstyle{definition}
\newtheorem{definition}{Definition}[section]
\newtheorem{example}[definition]{Example}\newtheorem{remark}[definition]{Remark}
\theoremstyle{plain}\newtheorem{theorem}[definition]{Theorem}\newtheorem{corollary}[definition]{Corollary}\newtheorem{proposition}[definition]{Proposition}

\DeclareMathOperator{\M}{\mathcal{M}}
\DeclareMathOperator{\s}{\mathbb{S}}
\DeclareMathOperator{\G}{\mathcal{G}}
\DeclareMathOperator{\R}{\mathbb{R}}

\DeclareMathOperator{\wce}{wce}
\DeclareMathOperator{\rank}{rank}
\DeclareMathOperator{\dist}{dist}

\DeclareMathOperator{\trace}{trace}
\DeclareMathOperator{\sym}{sym}
\DeclareMathOperator*{\argmin}{arg\,min}

\begin{document}
%
\title{Optimal Monte Carlo integration on closed manifolds}


\author{Martin Ehler}
\address{University of Vienna, 
Department of Mathematics, 
Vienna, Austria}
\email{Email: martin.ehler@univie.ac.at}
\author{Manuel Gr\"af}
\address{
Austrian Academy of Sciences,
Acoustics Research Institute, 
Vienna, Austria}
\email{Email: mgraef@kfs.oeaw.ac.at}
\author{Chris. J. Oates}
\address{Newcastle University,
School of Mathematics, Statistics and Physics,
Newcastle-Upon-Tyne, UK}
\email{chris.oates@ncl.ac.uk}


%


\maketitle

\begin{abstract}
The worst case integration error in reproducing kernel Hilbert spaces of standard Monte Carlo methods with $n$ random points decays as $n^{-1/2}$. However, re-weighting of random points can sometimes be used to improve the convergence order. This paper contributes general theoretical results for Sobolev spaces on closed Riemannian manifolds, where we verify that such re-weighting yields optimal approximation rates up to a logarithmic factor. We also provide numerical experiments matching the theoretical results for some Sobolev spaces on the sphere $\s^2$ and on the Grassmannian manifold $\mathcal{G}_{2,4}$. Our theoretical findings also cover function spaces on more general sets such as the unit ball, the cube, and the simplex $\R^d$. 
\end{abstract}

%

\section{Introduction}
Many problems in statistics and the applied sciences require the numerical computation of integrals for an entire class of functions. Given $\M\subset\R^D$, endowed with some probability measure $\mu$, and a function $f:\M\rightarrow \R$, standard Monte Carlo methods approximate the integral 
$\int_{\M} f(x)\mathrm{d}\mu(x)$ by the finite sum 
\begin{equation}\label{eq:first one basic}
\frac{1}{n}\sum_{j=1}^n f(x_j),
\end{equation} 
where $\{x_j\}_{j=1}^n\subset\M$ are independent samples from $\mu$. On the one hand, Monte Carlo integration is widely used in many numerical and statistical applications \cite{Robert:2013}. It is well-known, however, that the expected worst case integration error for $n$ random points using \eqref{eq:first one basic} in reproducing kernel Hilbert spaces does not decay faster than $n^{-1/2}$, cf.~\cite{Brauchart:fk,Breger:2016rc,Hinrichs:2010ij,Nowak:2010rr,Plaskota:2009xu} and \cite[proof of Corollary 2.8]{Graf:2013zl}. To improve the approximation, it has been proposed to re-weight the random points \cite{Briol:2016dz,Oettershagen:2017,Rasmussen:2003wq,Sommariva:2006}, which is of particular importance when $\mu$ can only be sampled \cite{Oates:2017} and evaluating $f$ is rather expensive. 

That re-weighting of {\it deterministic} points can lead to optimal convergence order has been known since the pioneering work of \cite{Bakhvalov:1959}. 
For Sobolev spaces on the sphere and more generally on compact Riemannian manifolds, there are numerically feasible strategies to select deterministic points and weights matching optimal worst case error rates, cf.~\cite{Brandolini:2014oz,Brauchart:fk,Breger:2016vn}, see also \cite{Hellekalek:2016ax,Hinrichs:2015bf,Niederreiter:2003hb}.

The use of \emph{random} points avoids the need to manually specify a point set and can potentially lead to simpler algorithms if the geometry of the manifold $\mathcal{M}$ is complicated. For random points, it was derived in \cite{Briol:2016dz} that the optimal rate for $[0,1]^d$, the sphere, and quite general domains in $\R^d$ can be matched up to a logarithmic factor if the weights are optimized with respect to the underlying reproducing kernel. Decay rates of the worst case integration error for Sobolev spaces of dominating mixed smoothness on the torus and the unit cube were studied in \cite{Oettershagen:2017}. Numerical experiments on the Grassmannian manifold were provided in \cite{Ehler:2017zh}. We refer to \cite{Trefethen:2017gf,Trefethen:2017ek}, for further related results.

The present note is dedicated to verify that, for Sobolev spaces on closed Riemannian manifolds, random points with optimized weights yield optimal decay rates of the worst case error up to a logarithmic factor. We should point out that we additionally allow for the restriction to nonnegative weights, a desirable property not considered in \cite{Briol:2016dz}. Our findings also transfer to functions defined on more general sets such as the $d$-dimensional unit ball and the simplex. 

First, we bound the worst case error by the covering radius of the underlying points. Second, we use estimates on the covering radius of random points from \cite{Reznikov:2015zr}, see also \cite{Brauchart:2016pz} for the sphere, to establish the optimal approximation rate up to a logarithmic factor. Some consequences for popular Bayesian integration methods are then presented. Numerical experiments for the sphere and the Grassmannian manifold are provided that support our theoretical findings. We also discuss the extension to the unit ball, the cube, and the simplex. 

\section{Preliminaries}
Let $\mathcal{M}\subset\R^D$ be a smooth, connected, closed Riemannian manifold of dimension $d$, endowed with the normalized Riemannian measure $\mu$ throughout the manuscript. Prototypical examples for $\M$ are the sphere and the Grassmannian
\begin{align*}
\s^{d}&=\{x\in\R^{d+1} : \|x\|=1\},\\
\G_{k,m} & = \{x\in\R^{m\times m} : x^\top = x,\; x^2 = x,\; \rank(x)=k\},
\end{align*}
where $d = k(m-k)$ with $D=m^2$ in case of the Grassmannian.

Let $\mathcal{H}$ be any normed space of continuous functions $f:\M\rightarrow \R$. For points $\{x_j\}_{j=1}^{n}\subset \mathcal{M}$ and weights
$\{w_j\}_{j=1}^{n} \subset \R$, the worst case error of integration is defined by 
\begin{equation}\label{eq:wce ee}
\wce(\{(x_j, w_j)\}_{j=1}^{n},\mathcal{H}) :=\sup_{\substack{f\in \mathcal{H}\\ \|f\|_{\mathcal{H}}\leq 1}}\Big| \int_{\mathcal{M}} f(x)\mathrm{d}\mu(x)- \sum_{j=1}^{n}  w_j f(x_j) \Big|.
\end{equation}
Suppose now that $\mathcal{H}$ is a reproducing kernel Hilbert space, denoted by $\mathcal{H}_K$, then the squared worst case error can be expressed in terms of the reproducing kernel $K$ by
\begin{align}\label{eq:wce in K}
\begin{split}
\wce(\{(x_j, w_j)\}_{j=1}^{n},\mathcal{H}_K)^2 &= \sum_{i,j=1}^n  w_i w_{j}K(x_i,x_j) - 2  \sum_{j=1}^{n}  w_{j} \int_{\mathcal{M}}K(x_j,y)\mathrm{d}\mu(y)\\
&\qquad +\int_{\mathcal{M}}\int_{\mathcal{M}}K(x,y)\mathrm{d}\mu(x)\mathrm{d}\mu(y).
\end{split}
\end{align}
If $x_1,\ldots,x_n\in\mathcal{M}$ are random points, independently distributed according to $\mu$, then it holds
\begin{equation}\label{eq:rand points}
  \sqrt{\mathbb{E} \Big[ \wce(\{(x_j,\tfrac{1}{n})\}_{j=1}^{n},\mathcal{H}_K)^2\Big]}  \asymp n^{-\frac{1}{2}},
\end{equation}
cf.~\cite{Brauchart:fk,Breger:2016rc,Nowak:2010rr} and \cite[Proof of Corollary 2.8]{Graf:2013zl}. Hence, even if $\mathcal{H}_K$ consists of arbitrarily smooth functions, the left hand side of \eqref{eq:rand points} decays only like $n^{-\frac{1}{2}}$.

The present note is dedicated to the question if and, as the case may be, how much one can actually improve the error rate in \eqref{eq:rand points} when replacing the equal weights $1/n$ with weights $\{ w_j\}_{j=1}^n$ that are customized to the random points $\{x_j\}_{j=1}^n$.  

\section{Bounding the worst case error by the covering radius}
To define appropriate smoothness spaces, let $\Delta$ denote the Laplace-Beltrami operator on $\M$ and let $\{\varphi_\ell\}_{\ell=0}^\infty$ be the collection of its orthonormal
eigenfunctions with eigenvalues $\{-\lambda_\ell\}_{\ell=0}^\infty$ arranged by
$0=\lambda_0 \leq \lambda_1\leq \ldots$. We choose each $\varphi_\ell$, $\ell=0,1,2,\ldots$, to be real-valued with $\varphi_0\equiv 1$. Given $f\in L_p(\mathcal{M})$ with $1\leq p\leq \infty$, the Fourier transform is defined by 
\begin{equation*}
\hat{f}(\ell):=\int_{\mathcal{M}} f(x)\varphi_\ell(x)\mathrm{d}\mu(x),\qquad \ell=0,1,2,\ldots,
\end{equation*}
with the usual extension to distributions on $\mathcal{M}$. For $1\leq p\leq \infty$ and $s>0$, the Sobolev space $H^s_p(\mathcal{M})$ is the collection of all distributions on $\mathcal{M}$ with $(I-\Delta)^{s/2}f\in L_p(\mathcal{M})$, i.e., with 
\begin{equation}\label{eq:def norm sob}
\| f\|_{H^s_p}  := \|(I-\Delta)^{s/2} f \|_{L_{p}} = \|\sum_{\ell=0}^\infty  (1+\lambda_\ell)^{s/2} \hat{f}(\ell) \varphi_\ell \|_{L_p}<\infty.
\end{equation}
For $s>d/p$, each function in $H^s_p(\mathcal{M})$ is continuous, cf.~\cite{Brandolini:2014oz} and \cite[Theorem 7.4.5, Section 7.4.2]{Triebel:1992aa}, so that point evaluation makes sense.

For $s>d/p$ and any set of points $\{x_j\}_{j=1}^n\subset\M$ with arbitrary weights $\{ w_j\}_{j=1}^n\subset\R$, we have 
\begin{equation}\label{eq:Brandi}
n^{-s/d} \lesssim \wce(\{(x_j, w_j)\}_{j=1}^{n},H_p^s(\M)),
\end{equation}
see \cite{Brauchart:fk} for the sphere and \cite{Brandolini:2014oz} for the general case. Note that the constant in \eqref{eq:Brandi} may depend on $s$, $\M$, and $p$. 

Another lower bound involves the covering radius, 
\begin{equation*}
\rho_n : = \max_{x\in\M} \min_{j=1,\ldots,n} \dist_{\M}(x,x_j),
\end{equation*}
where $\dist_{\M}$ denotes the geodesic distance. According to \cite{Breger:2016rc}, it also holds
\begin{equation}\label{eq:b r 1}
\rho_n^{s+d/q} \lesssim \wce(\{(x_j, w_j)\}_{j=1}^{n},H^s_p(\mathcal{M})),
\end{equation}
where $1/p+1/q=1$.

Attempting to match this lower bound, we shall optimize the weights. Given points $\{x_j\}_{j=1}^n\subset\M$, we define optimal weights with nonnegativity constraints by
\begin{equation}\label{eq:opt weights def 2}
\{\widehat{ w}^{\geq 0;\,p}_j\}_{j=1}^n:= \argmin_{ w_1,\ldots, w_n\geq 0}\wce(\{(x_j, w_j)\}_{j=1}^{n},H_p^s(\M)).
\end{equation}
The worst case error for the optimized weights is upper bounded by the covering radius:
\begin{theorem}\label{th:fundamentals 2}
Let $1\leq p\leq\infty$, suppose $s>d/p$, and let $\{x_j\}_{j=1}^n\subset\M$ be any set of points with covering radius $\rho_n$. Then the optimized weights $\{\widehat{ w}^{\geq 0;\, p}_j\}_{j=1}^n$ in \eqref{eq:opt weights def 2} satisfy
\begin{equation}\label{eq:in theorem 2}
\wce(\{(x_j,\widehat{ w}^{\geq 0;\,p}_j)\}_{j=1}^{n},H_p^s(\M)) \lesssim \rho_n^s.
\end{equation}
\end{theorem}
Note that the constant in \eqref{eq:in theorem 2} may depend on $\M$, $s$, and $p$. 
\begin{remark}\label{rem:olala}
If we fix a constant $c>0$, independent of $n$ and $\{x_j\}_{j=1}^n$, then any weights $\{\widetilde{ w}^p_j\}_{j=1}^n\subset\R$ with
\begin{equation*}
\wce(\{(x_j,\widetilde{ w}^{p}_j)\}_{j=1}^{n},H_p^s(\M))\leq c \cdot \wce(\{(x_j,\widehat{ w}^{\geq 0;\,p}_j)\}_{j=1}^{n},H_p^s(\M))
\end{equation*}
satisfy the estimate
\begin{equation}\label{eq:in theorem 22}
\wce(\{(x_j,\widetilde{ w}^{p}_j)\}_{j=1}^{n},H_p^s(\M)) \lesssim \rho_n^s.
\end{equation}
This fact is beneficial when we compute weights numerically.
\end{remark}
\begin{proof}[Proof of Theorem \ref{th:fundamentals 2}]
Let $X:=\{x_j\}_{j=1}^n$ and $\rho(X):=\rho_n$. There is a subset $Y=\{y_j\}_{j=1}^m\subset X$ with covering radius $\rho(Y)\leq 2\rho(X)$ and minimal separation 
\begin{equation*}
\delta(Y):=\min_{\substack{a,b\in Y\\ a\neq b}} \dist_{\M}(a,b)
\end{equation*}
such that $\rho(Y)\leq 2\delta(Y)$, cf.~\cite[Section 3]{Filbir:2010aa}. We observe that our present setting satisfies the technical requirements of \cite{Filbir:2010aa}, cf.~\cite{Hsu:1999oh} and \cite[page 159]{Chavel:1984cr}. We deduce from \cite[Lemma 2.14]{Brandolini:2014oz} and \cite[Theorem 3.1]{Filbir:2010aa}, see also \cite{Mhaskar:2002ys} for $\M=\s^d$, with \cite[Corollary 2.15]{Brandolini:2014oz} that there exist $ w_1,\ldots, w_m\gtrsim \rho_n^d$, such that
\begin{equation*}
\wce(\{(y_j, w_j)\}_{j=1}^{m},H_p^s(\M)) \lesssim m^{-s/d}. 
\end{equation*}
Since $\wce(\{(x_j,\widehat{ w}^{\geq 0;\, p}_j)\}_{j=1}^{n},H_p^s(\M))\leq \wce(\{(y_j, w_j)\}_{j=1}^{m},H_p^s(\M))$, we also obtain 
\begin{equation*}
\wce(\{(x_j,\widehat{ w}^{\geq 0;\, p}_j)\}_{j=1}^{n},H_p^s(\M)) \lesssim m^{-s/d}. 
\end{equation*}
The general lower bound on the covering radius $m^{-1/d}\lesssim \rho(Y)$, cf.~\cite{Breger:2016rc}, implies 
\begin{equation*}
m^{-s/d}\lesssim \rho(Y)^s\lesssim \rho(X)^s,
\end{equation*}
which concludes the proof.
\end{proof}
Combining \eqref{eq:b r 1} with Theorem \ref{th:fundamentals 2} for $p=1$ yields 
\begin{equation}\label{eq:in theorem 3}
\wce(\{(x_j,\widehat{ w}^{\geq 0;\,1}_j)\}_{j=1}^{n},H_1^s(\M)) \asymp \rho_n^s,
\end{equation}
so that the worst case error's asymptotic behavior is governed by the covering radius.

\begin{remark}
The above proof reveals that in the setting of Theorem \ref{th:fundamentals 2} there exist $\{ w_j\}_{j=1}^n$ with either $ w_j\gtrsim \rho^d_n$ or $ w_j=0$, for $j=1,\ldots,n$, such that 
\begin{equation*}
\wce(\{(x_j, w_j)\}_{j=1}^{n},H_p^s(\M)) \lesssim \rho_n^s.
\end{equation*}
\end{remark}
The covering radius $\rho_n$ of any $n$ points in $\M$ is lower bounded by $\gtrsim n^{-1/d}$, which follows from standard volume arguments.  
If $\{x_j\}_{j=1}^n\subset\M$ are points with asymptotically optimal covering radius, i.e., $\rho_n\asymp  n^{-1/d}$, then 
Theorem \ref{th:fundamentals 2} yields the optimal rate for the worst case integration error
\begin{equation}\label{eq:in theorem 3}
\wce(\{(x_j,\widehat{ w}^{\geq 0;\, p}_j)\}_{j=1}^{n},H_p^s(\M)) \lesssim  n^{-s/d},
\end{equation}
cf.~\eqref{eq:Brandi}. 

Several point sets on $\s^2$ with asymptotically optimal covering radius are discussed in \cite{Hardin:2016fv}, see quasi-uniform point sequences therein, and see \cite{Breger:2016rc} for general $\M$. The covering radius of random points is studied in \cite{Brauchart:2016pz,Oates:2017mw,Reznikov:2015zr}, which leads to almost optimal bounds on the worst case error in the subsequent section. Although we shall consider independent random points, it is noteworthy that it is verified in \cite{Oates:2017mw} that the required estimates on the covering radius still hold for random points arising from a Markov chain instead of being independent. 
Note also that results related to Theorem \ref{th:fundamentals 2} are derived in \cite{Mhaskar:2017nr} for more general spaces $\M$.

\section{Consequences for random points}\label{sec:random points}
For random points $\{x_j\}_{j=1}^n\subset\M$ and any weights $\{ w_j\}_{j=1}^n\subset\R$, no matter if random or not, \eqref{eq:Brandi} implies, for all $r>0$, 
\begin{equation}\label{eq:opt est below}
n^{-s/d}\lesssim \Big(\mathbb{E}\big[\wce(\{(x_j, w_j)\}_{j=1}^{n},H_p^s(\M))^r\big]\Big)^{1/r},
\end{equation}
where the constant may depend on $s$, $p$,  and $\M$. Note that if $\{x_j\}_{j=1}^n\subset\M$ are random points, then the weights $\{\widehat{ w}^{\geq 0;\, p}_j\}_{j=1}^n$ are random as well. We shall deduce that Theorem \ref{th:fundamentals 2} implies that the optimal worst case error rate is (almost) matched in these cases:
\begin{corollary}\label{th:only one}
Let $\{x_j\}_{j=1}^n\subset\M$ be random points, independently distributed according to $\mu$. Suppose $1\leq p\leq\infty $ and $s>d/p$, then, for each $r\geq 1/s$, it holds
\begin{equation}\label{optimality}
\Big(\mathbb{E}\big[\wce(\{(x_j,\widehat{ w}^{\geq 0;\, p}_j)\}_{j=1}^{n},H_p^s(\M))^r\big]\Big)^{1/r}  \lesssim n^{-s/d}\log(n)^{s/d}.
\end{equation}
\end{corollary}
Note that Corollary \ref{th:only one} yields the optimal rate up to the logarithmic factor $\log(n)^{s/d}$, cf.~\eqref{eq:opt est below}, and that the constant in \eqref{optimality} may depend on $s$, $\M$, $p$, and $r$.
\begin{proof}
From \cite[Theorem 3.2, Corollary 3.3]{Reznikov:2015zr} we deduce that, for each $r\geq 1$, 
\begin{equation}\label{eq:est cov radius from saff}
\big(\mathbb{E}\rho^r_n\big)^{1/r} \asymp n^{-1/d}\log(n)^{1/d},
\end{equation}
where the constant may depend on $\M$ and $r$. Thus, Theorem \ref{th:fundamentals 2} implies 
\begin{equation*}
\mathbb{E}\big[\wce(\{(x_j,\widehat{ w}^{\geq 0;\, p}_j)\}_{j=1}^{n},H_p^s(\M))^r\big]\lesssim \big(\tfrac{\log(n)}{n}\big)^{sr/d},
\end{equation*}
for each $r\geq 1/s$. 
\end{proof}
\begin{remark}
Let $\nu$ be a probability measure on $\M$ that is absolutely continuous with respect to $\mu$ and its density is bounded away from zero, i.e., $\nu=f \mu$ with $f(x)\geq c>0$, for all $x\in\M$. Corollary \ref{th:only one} still holds for independent samples from $\nu$, where the constant in \eqref{optimality} then also depends on $c$. This is due to $2n/c$ independent samples from $\nu$ covering $\M$ at least as good as $n$ independent samples from $\mu$. 
\end{remark}

Corollary \ref{th:only one} yields bounds on the moments of the worst case integration error. The results in \cite{Reznikov:2015zr} also enable us to derive probability estimates:
\begin{corollary}\label{co 22}
Under the assumptions of Corollary \ref{th:only one}, there are positive constants $c_1,\ldots,c_4$ depending on $\M$ and where $c_2$ may additionally depend on $s$ and $p$, such that, 
\begin{equation*}
\mathbb{P}\Big(\wce(\{(x_j,\widehat{ w}^{\geq 0;\, p}_j)\}_{j=1}^{n},H_p^s(\M))  \geq c_2 \big(\tfrac{r\log(n)}{n}\big)^{s/d} \Big)\leq c_3 \big(\tfrac{1}{n}\big)^{c_4r-1},
\end{equation*}
for all $r\geq c_1$.
\end{corollary}
\begin{proof}
By applying Theorem \ref{th:fundamentals 2} we deduce that there is a constant $c>0$, which may depend on $\M$, $s$, and $p$, such that
\begin{equation*}
\wce(\{(x_j,\widehat{ w}^{\geq 0;\, p}_j)\}_{j=1}^{n},H_p^s(\M)) \leq c \rho_n^s.
\end{equation*}
According to \cite[Theorem 2.1]{Reznikov:2015zr}, there are constants $c_1,\tilde{c}_2, c_3,c_4>0$, which may depend on $\M$, such that, for all $r\geq c_1$, 
\begin{equation*}
\mathbb{P}\Big( \rho_n\geq \tilde{c}_2\big(\tfrac{r\log(n)}{n}\big)^{1/d} \Big)\leq c_3 \big(\tfrac{1}{n}\big)^{c_4r-1}.
\end{equation*}
Raising the left inequality to the power $s$ and multiplying by $c$ yields the desired result with $c_2:=c\tilde{c}^s_2$. 
\end{proof}
\begin{remark}\label{remark:here and there}
Corollary \ref{th:only one} and Corollary \ref{co 22} also hold for weights $\{\widetilde{ w}_j^p\}_{j=1}^n$ that minimize \eqref{eq:opt weights def 2} up to a constant factor as discussed in Remark \ref{rem:olala}, and, in particular, for the unconstrained minimizer 
\begin{equation}\label{eq:opt weights def}
\{\widehat{ w}^p_j\}_{j=1}^n:= \argmin_{\{ w_j\}_{j=1}^n\subset \R}\wce(\{(x_j, w_j)\}_{j=1}^{n},H_p^s(\M)).
\end{equation}
Nonnegative weights are often more desirable for numerical applications of cubatures points, but solving the constrained minimization problem \eqref{eq:opt weights def 2} is usually more involved than dealing with the unconstrained problem \eqref{eq:opt weights def}.

\end{remark}

\section{Relations to Bayesian Monte Carlo}
Our results have consequences for {\it Bayesian cubature}, cf.~\cite{Larkin:1972}, an integration method whose output is not a scalar but a distribution. Bayesian cubature enables a statistical quantification of integration error, useful in the context of a wider computational work-flow to measure the impact of integration error on subsequent output, cf.~\cite{Briol:2016dz,Cockayne:2017}.

Consider a linear topological space $\mathcal{L}$ of continuous functions on $\mathcal{M}$. The integrand $f$ in Bayesian cubature is treated as a Gaussian random process; that is, $f : \mathcal{M} \times \Omega \rightarrow \mathbb{R}$, where $f(\cdot,\omega) \in \mathcal{L}$ for each $\omega \in \Omega$, and the random variables $\omega \mapsto L f(\cdot,\omega) \in \mathcal{L}$ are (univariate) Gaussian for all continuous linear functionals $L$ on $\mathcal{L}$, such as integration ($\mathcal{I} f = \int_{\mathcal{M}} f(x) \mathrm{d} \mu(x)$) and point evaluation ($\delta_x f = f(x)$) operators, cf.~[2]. The Bayesian approach is then taken, wherein the process $f$ is constrained to interpolate the values $\{(x_j,f(x_j))\}_{j=1}^n$. Formally, this is achieved by conditioning the process on the data provided through the point evaluation operators $\delta_{x_j}(f) = f(x_j)$, for $\{x_j\}_{j=1}^n \subset \mathcal{M}$. The conditioned process, denoted $f_n$, is again Gaussian (cf.~[2]) and as such the linear functional $\mathcal{I}  f_n$ is a (univariate) Gaussian; this is the output of the Bayesian cubature method. This distribution, defined on the real line, provides statistical uncertainty quantification for the (unknown) true value of the integral.


Concretely, let $K(x,y) = \text{cov}(f(x),f(y))$ denote the covariance function that characterizes the Gaussian probability model. The output of Bayesian cubature is the univariate Gaussian distribution with mean 
\begin{equation}\label{eq:sder}
\big(f(x_1),\ldots,f(x_n)  \big) {\underbrace{\left[ \begin{array}{ccc} K(x_1,x_1) & \dots & K(x_1,x_n) \\ \vdots & & \vdots \\ K(x_n,x_1) & \dots & K(x_n,x_n) \end{array} \right]}_{\mathcal{K}}}^{-1}  \underbrace{\left[ \begin{array}{c} \int_{\M} K(x_1,y) \mathrm{d}\mu(y) \\ \vdots \\ \int_{\M} K(x_n,y) \mathrm{d}\mu(y) \end{array} \right]}_{b},
\end{equation}
cf.~\cite{Briol:2016dz}. 
This expression is recognized as a weighted integration method with weights $\widehat{w} = (\widehat{w}_1,\dots,\widehat{w}_n)^\top$ implicitly defined by $\widehat{w}:=\mathcal{K}^{-1}b$, so that \eqref{eq:sder} becomes
\begin{equation*}
\sum_{j=1}^n \widehat{w}_j f(x_j).
\end{equation*}
Any symmetric positive definite covariance function $K$ can be viewed as a reproducing kernel. In particular, the Bessel kernel 
\begin{equation}\label{eq:Bessel def}
K^{(s)}_{B}(x,y)=\sum_{\ell=0}^\infty (1+\lambda_\ell)^{-s} \varphi_\ell(x)\varphi_\ell(y),\quad x,y\in\M,
\end{equation}
reproduces $H^s(\M):=H^s_2(\M)$, for $s>d/2$. Observe that the weights $\widehat{w}$ just defined solve the unconstrained minimization problem \eqref{eq:opt weights def} for $p=2$. 
The latter follows from the quadratic minimization form in \eqref{eq:wce in K} as well as from the posterior mean being an $L_2$-optimal estimator \cite{Kimeldorf:1970}. 

The variance of the Gaussian measure can be shown to be formally equal to \eqref{eq:wce in K} when these weights are substituted, see \cite{Briol:2016dz}. The special case where the points $\{x_j\}_{j=1}^n$ are random was termed {\it Bayesian Monte Carlo} in \cite{Rasmussen:2003wq}. Therefore, our results in Section \ref{sec:random points} have direct consequences for Bayesian Monte Carlo. Due to Remark \ref{remark:here and there} within this Bayesian setting, Corollary \ref{th:only one} and Corollary \ref{co 22} generalize earlier work of \cite{Briol:2016dz} to a general smooth, connected, closed Riemannian manifold.

\section{Numerical experiments for the sphere and the Grassmannian}
Numerically computing the worst case error $\wce(\{(x_j, w_j)\}_{j=1}^{n},H_p^s(\M))$ is difficult in general but, for $p=2$, it is expressed in terms of the reproducing kernel in \eqref{eq:wce in K}. Therefore, our numerical experiments are designed for $p=2$. However, the kernel $K^{(s)}_B$ itself, see \eqref{eq:Bessel def}, may still be difficult to evaluate numerically, so that we would like to allow for other kernels in numerical experiments. If $K$ is any positive definite kernel on $\M$ that reproduces $H^s(\M)$ with equivalent norms, then 
\begin{equation*}
\wce(\{(x_j, w_j)\}_{j=1}^{n},H^s(\M)) \asymp\wce(\{(x_j, w_j)\}_{j=1}^{n},\mathcal{H}_K).
\end{equation*}
Therefore, the asymptotic results in Corollary \ref{th:only one} and Corollary \ref{co 22} are the same when replacing $\{\widehat{ w}^{\geq 0;\, 2}_j\}_{j=1}^n$ with the minimizer
\begin{equation}\label{eq:opt mit kern}
\{\widehat{ w}_j^{\geq 0;K}\}_{j=1}^n:= \argmin_{ w_1,\ldots, w_n\geq 0}\wce(\{(x_j, w_j)\}_{j=1}^{n},\mathcal{H}_K).
\end{equation}
%
%
%
Dropping the nonnegativity constraints yields $\widehat{ w}^{K}$, which is given by $\widehat{w} = \mathcal{K}^{-1}b$, where $\mathcal{K}$ and $b$ are as in \eqref{eq:sder}. 
To provide numerical experiments for Sobolev spaces on the sphere $\s^2\subset\R^3$ and on the Grassmannian $\mathcal{G}_{2,4}$, we shall specify suitable kernels in the following. We shall consider two kernels $K_1,K_2$ on the sphere $\s^2$ and two kernels $K_3,K_4$ on the Grassmannian $\mathcal{G}_{2,4}$. 

The numerical results are produced by taking sequences of random points $\{x_j\}_{j=1}^n$ with increasing cardinality $n$. We compute each of the three worst case errors
\begin{equation*}
\wce(\{(x_j,\tfrac{1}{n})\}_{j=1}^{n},\mathcal{H}_{K_i}),\quad \wce(\{(x_j,\widehat{ w}^{K_i}_j)\}_{j=1}^{n},\mathcal{H}_{K_i}), \quad\wce(\{(x_j,\widehat{ w}^{\geq 0;K_i}_j)\}_{j=1}^{n},\mathcal{H}_{K_i}),
\end{equation*}
for $i=1,\ldots,4$, and averaged these results over 20 instantiations of the random points. The constrained minimization problem for the latter two quantities is solved by using the Python CVXOPT library. It should be mentioned that numerical experiments on the sphere for the unconstrained optimizer $\widehat{ w}^{K_1}$ are also contained in \cite{Briol:2016dz}. 

The kernel
\begin{equation*}
K_{1}(x,y)  := 2 - \|x-y\|,\quad x,y\in\s^2,
\end{equation*}
reproduces the Sobolev space $H^{3/2}(\s^2)$ with an equivalent norm, cf.~\cite[Section 6.4.1]{Graf:2013zl}. To compute \eqref{eq:wce in K} and \eqref{eq:sder}, it is sufficient to notice
\begin{equation*}
\int_{\mathbb S^2} K_1(x,y) \mathrm{d}\mu(y) = \frac{2}{3}, \quad\text{for all } x\in \s^2.
\end{equation*}

By plotting the worst case error versus the number of points in a logarithmic scale, we are supposed to observe lines whose slopes coincide with the decay rate $-s/d$ for the optimized weights and slope $-1/2$ for the weights $1/n$. Indeed, we see in Figure \ref{subfig:1} that $\wce(\{(x_j,\tfrac{1}{n})\}_{j=1}^{n},\mathcal{H}_{K_1})$ for random points matches the error rate $-1/2$ predicted by \eqref{eq:rand points} with $d=2$. When optimizing the weights, we observe the decay rate $-3/4$ for both optimizations, $\widehat{ w}^{\geq 0;K}$ in \eqref{eq:opt mit kern} and the unconstrained minimizer $\widehat{ w}^K$. Hence, the numerical results match the rate predicted by the theoretical findings in \eqref{eq:opt est below}, \eqref{optimality} with $p=2$ and $r=1$. The logarithmic factor in \eqref{optimality} is not visible.

\begin{figure}
\subfigure[Optimized weights for $K_1$ yield decay $-3/4$]{
 \includegraphics[width=.45\textwidth]{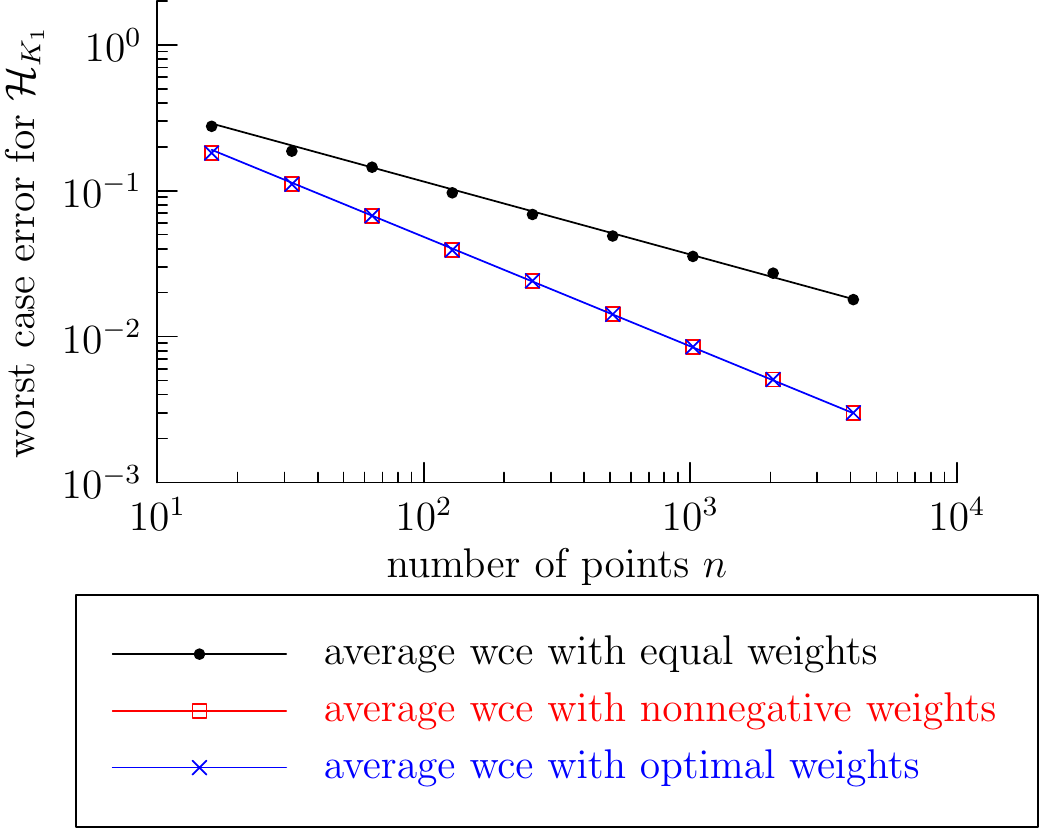}\label{subfig:1}
 }
 \hfill
 \subfigure[Fixed weights $1/n$ for $K_2$ are stuck with decay $-1/2$]{
 \includegraphics[width=.45\textwidth]{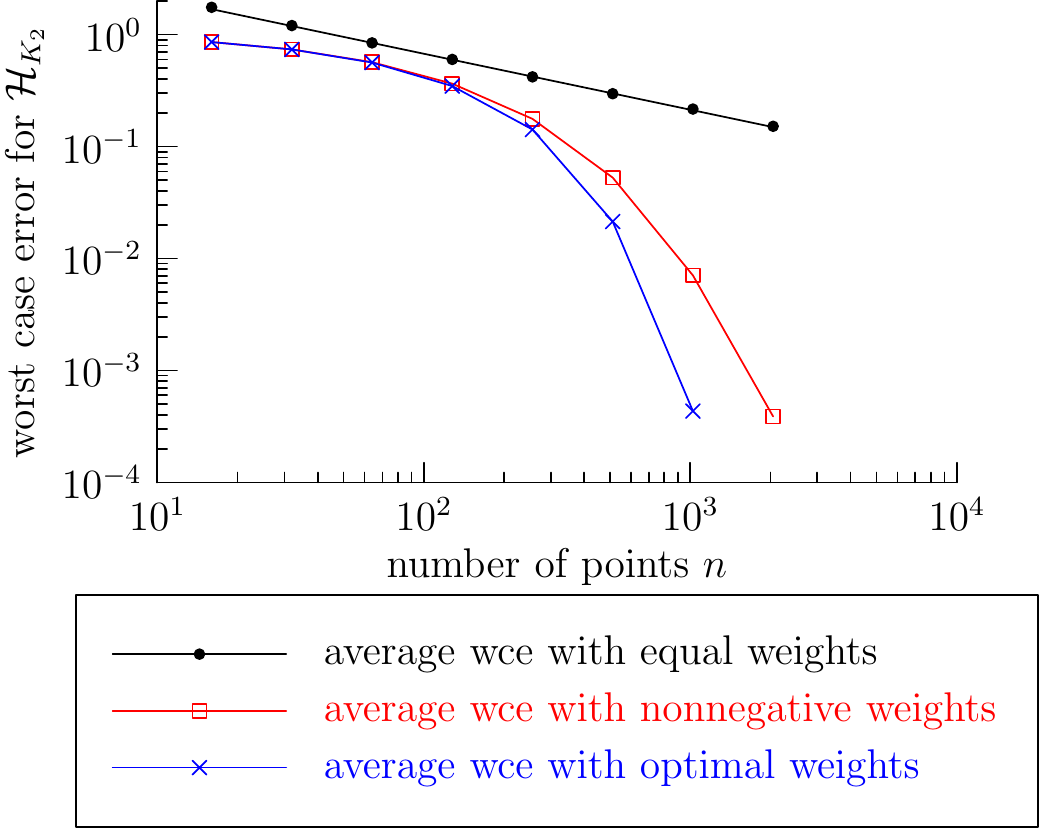}\label{subfig:2}
}
\caption{The worst case integration error for $\mathcal{H}_{K_1}$ and $\mathcal{H}_{K_2}$ averaged over $20$ instances of random points in logarithmic scalings. The lines with exact slope $-1/2$ and $-3/4$ are adjusted to approximate the data points.}\label{fig:1}
\end{figure}

The smooth kernel 
\begin{equation*}    
K_2 (x,y) := 48\exp(-12\|x-y\|^2),\quad x,y\in\s^2,
\end{equation*}
generates a space $\mathcal{H}_{K_2}$ of smooth functions contained in $H^s(\s^2)$, for all $s>0$, and satisfies
\begin{equation*}
\int_{\mathbb S^2} K_2(x,y) \mathrm{d}\mu(y) = 1 - \mathrm \exp(-48),\quad\text{for all } x\in\s^2.
\end{equation*}
 Our numerical experiments in Figure \ref{subfig:2} suggest that the decay rate for the optimized weights is indeed beyond linear. Note that the equal weight case is stuck with the decay rate $-1/2$, although we are now dealing with arbitrarily smooth functions. 

The dimension of the Grassmannian $\mathcal{G}_{2,4}$ is $d=4$, and we consider the two reproducing kernels
\begin{align*}    
K_{3}(x,y) & := \sqrt{(2-\trace(xy))^{3}} + 2 \trace(xy), \\
K_4 (x,y) & := \tfrac{3}{2}\exp(\trace(xy)-2)).
\end{align*}
Note that $K_3$ reproduces $H^{7/2}(\mathcal{G}_{2,4})$ with an equivalent norm, and $\mathcal{H}_{K_4}$ is contained in $H^{s}(\G_{2,4})$, for all $s>0$. The terms \eqref{eq:wce in K} and \eqref{eq:sder} are computable from 
\begin{align*}
\int_{\mathcal G_{2,4}} K_3(x,y) \mathrm{d}\mu(y) & = 2 + \frac{74}{75}\sqrt2-\frac{2}{5}\log(1+\sqrt2),\\
\int_{\mathcal G_{2,4}} K_4(x,y) \mathrm{d}\mu(y) &= \frac{3}{2} \mathrm \exp(-1) \int_0^1 \frac{\mathrm{sinh}(t)}{t}\mathrm{d}t.
\end{align*}
for all $x\in\mathcal{G}_{2,4}$. 

We observe in Figure \ref{subfig:3} that the random points with equal weights yield decay rate $-1/2$ and optimizing weights leads to $-7/8$ matching the optimal rate in \eqref{eq:opt est below}, \eqref{optimality} with $d=4$. 
\begin{figure}
\subfigure[Optimized weights for $K_3$ yield decay $-7/8$]{
  \includegraphics[width=.45\textwidth]{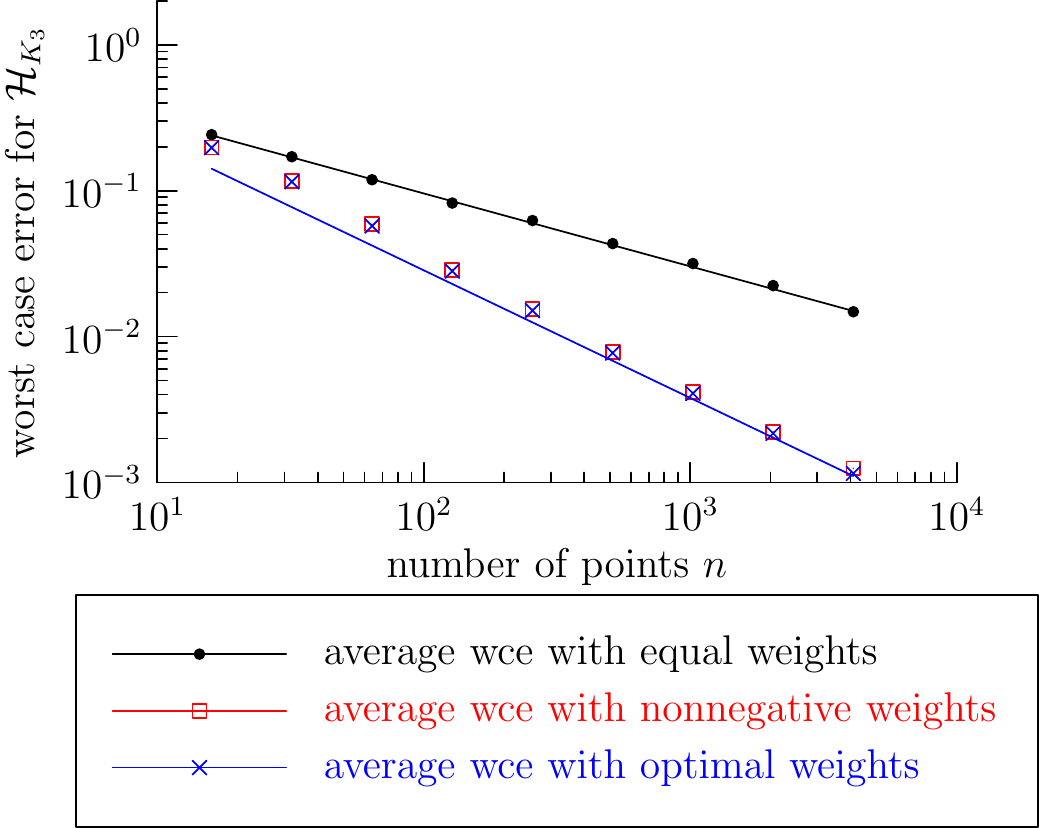}\label{subfig:3}
 }
 \hfill
 \subfigure[Fixed weights $1/n$ for $K_4$ are stuck with decay $-1/2$]{
 \includegraphics[width=.45\textwidth]{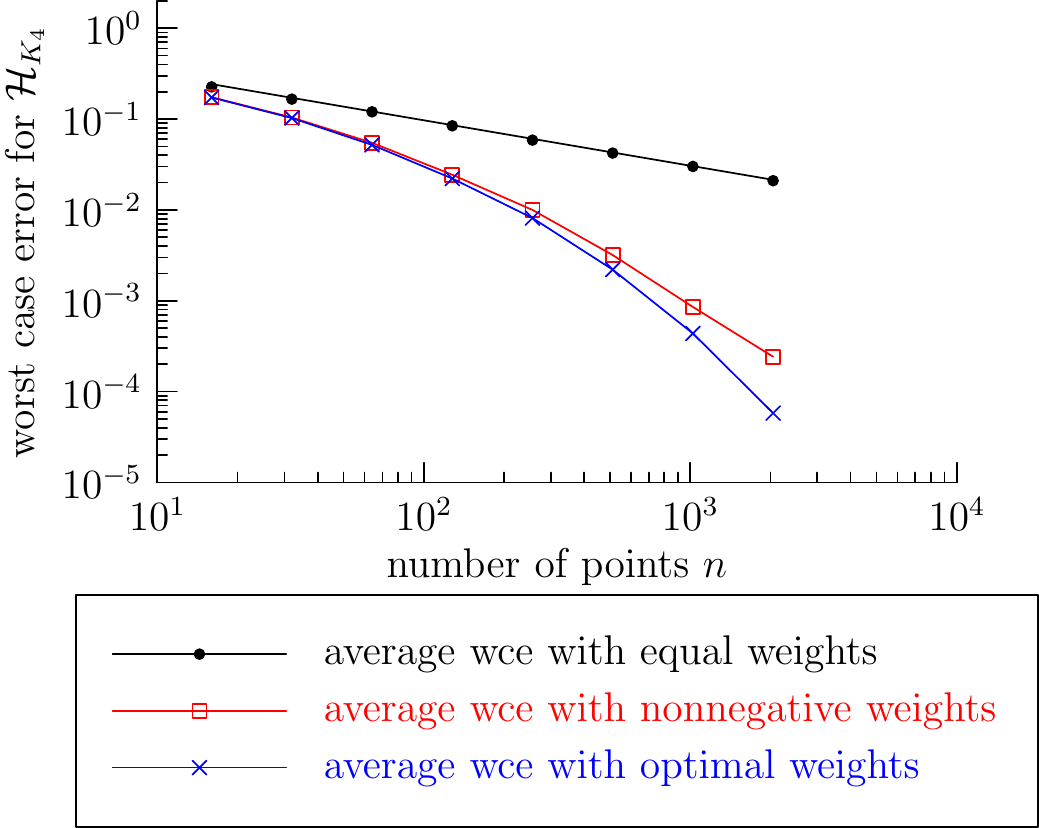}\label{subfig:4}
}
\caption{The worst case integration error for $\mathcal{H}_{K_3}$ and $\mathcal{H}_{K_4}$ averaged over $20$ instances with random points in logarithmic scalings.}\label{fig:2}
\end{figure}
In Figure \ref{subfig:4}, it seems that the worst case error for $\mathcal{H}_{K_4}$  decays faster than linear when optimizing the weights for random points on the Grassmannian $\mathcal{G}_{2,4}$ outperforming the equal weight case with its rate $-1/2$.

\section{Beyond closed manifolds}
We shall make use of the push-forward to transfer our results on the worst case integration error from closed manifolds to more general sets. Suppose $S$ is a topological space and $h:\mathcal{M}\rightarrow S$ is Borel measurable and surjective. We endow $S$ with the push-forward measure $h_*\mu$ defined by $(h_*\mu)(A)=\mu(h^{-1}A)$ for any Borel measurable subset $A\subset S$. By abusing notation, let $\dist_{\M}(A,B):=\inf_{a\in A;\, b\in B} \dist_{\M}(a,b)$ for $A,B\subset \M$, and we put
\begin{equation}\label{m dis}
\dist_{S,h}(x,y):= \dist_{\M}(h^{-1}x,h^{-1}y),\quad x,y\in S.
\end{equation} 
For $s>d/p$,  we define
\begin{equation}\label{eq:space def pp}
H^s_p(S)_h:=\{f:S\rightarrow\R : h^*f \in H^s_p(\mathcal{M})\}
\end{equation}
with $\|f\|_{H^s_p(S)_h}:=\|h^* f\|_{H^s_p(\M)}$, where $h^*f$ denotes the pullback $f\circ h$. This enables us to formulate the analogue of Theorem \ref{th:fundamentals 2}:
\begin{theorem}\label{th:pullbacky}
Given $1\leq p\leq\infty$ with $s>d/p$ and $\{x_j\}_{j=1}^n\subset S$, suppose that the following two conditions are satisfied,
\begin{itemize}
\item[a)] $\# h^{-1}x$ is finite, for all $x\in S$,
\item[b)] $\dist_{\mathcal{M}}(\{a\},h^{-1}x) \asymp \dist_{\M} (h^{-1}h(a),h^{-1}x)$,
for all $a\in\M$, $x\in S$.
\end{itemize}
Then there are nonnegative weights $\{ w_j\}_{j=1}^n$ such that 
\begin{equation*}
\wce(\{(x_j, w_j)\}_{j=1}^{n},H^s_p(S)_h) \lesssim \rho_n^s,
\end{equation*}
where $\rho_n$ denotes the covering radius of $\{x_j\}_{j=1}^n$ taken with respect to \eqref{m dis}.
\end{theorem}
Note that \eqref{m dis} is a quasi-metric on $S$ if the assumptions in Theorem \ref{th:pullbacky} are satisfied, i.e., the conditions of a metric are satisfied except for the triangular inequality that still holds up to a constant factor. 
\begin{proof}
Denote $\{z_{j,i}\}_{i=1}^{n_j}=h^{-1}x_j$, for $j=1\ldots,n$. According to Theorem \ref{th:fundamentals 2}, there exist nonnegative weights $\{ w_{j,i}\}_{i=1}^{n_j}$, for $j=1\ldots,n$, such that, for all $f\in H^s_p(S)_h$, 
\begin{equation}\label{eq:funda}
\Big|  \int_{\M} (h^*f) (z)\mathrm{d}\mu(z) - \sum_{j=1}^n\sum_{i=1}^{n_j}  w_{j,i}(h^*f) (z_{j,i}) \Big| \lesssim \rho_{\M}(\bigcup_{j=1}^n h^{-1}x_j)^s \|h^* f \|_{H^s_p(\M)},
\end{equation}
where $\rho_{\M}(\bigcup_{j=1}^n h^{-1}x_j)$ denotes the covering radius of $\bigcup_{j=1}^n h^{-1}x_j\subset\M$. The assumptions imply $\rho_n\asymp \rho_{\M}(\bigcup_{j=1}^n h^{-1}x_j)$, so that $ w_j:=\sum_{i=1}^{n_j}  w_{j,i}$, $j=1,\ldots,n$, and \eqref{eq:funda} lead to 
\begin{equation*}
\Big|  \int_{S} f(x) \mathrm{d}(h_*\mu)(x) - \sum_{j=1}^n  w_{j} f(x_j) \Big| \lesssim \rho_{n}^s \|f\|_{H^s_p(S)_h},
\end{equation*}
which concludes the proof.
\end{proof}
\begin{remark}
Since independent random points $\{x_j\}_{j=1}^n$ distributed according to $h_*\mu$ on $S$ with covering radius $\rho_n$ are generated by independent random points $\{z_j\}_{j=1}^n$ with respect to $\mu$ on $\M$ with $x_j=h(z_j)$, for $j=1,\ldots,n$, the observation 
\begin{equation*}
\rho_n\lesssim \rho_{\M}(\{z_j\}_{j=1}^n)
\end{equation*}
implies that also Corollary \ref{th:only one}, and Corollary \ref{co 22} hold for $H^s_p(S)_h$ and $h_*\mu$. 
\end{remark}

The impact of Theorem \ref{th:pullbacky} depends on whether or not the choices of $h$ yield reasonable function spaces $H^s_p(S)_h$, distances $\dist_{S,h}$, and measures $h_*\mu$. For instance, if $h$ is also injective with measurable $h^{-1}$, then $H^s(S)_h$ is the reproducing kernel Hilbert space with kernel 
\begin{equation*}
\sum_{\ell=0}^\infty (1+\lambda_\ell)^{-s} \psi_\ell(x)\psi_\ell(y),
\end{equation*}
where $\psi_\ell:=\varphi_\ell \circ h^{-1}$, so that $\{\psi_\ell\}_{\ell=0}^\infty$ is an orthonormal basis for the square integrable functions with respect to $h_*\mu$. In the following, we shall discuss a few special cases, in which $h$ is not injective. By using the results in \cite{Xu:1998rt} and \cite{Xu:2001fp}, we shall determine $H^s(S)_h$ for $S$ being the unit ball $\mathbb{B}^d:=\{x\in\R^d : \|x\|\leq 1\}$, the cube $[-1,1]^d$, and the simplex $\Sigma^d:=\{x\in\R^d: x_1,\ldots,x_d\geq 0;\; \sum_{i=1}^dx_i\leq 1\}$. 

Let $h:\s^d\rightarrow \mathbb{B}^d$ be the projection onto the first $d$ coordinates, i.e., $h(x)=(x_1,\ldots,x_d)\in \mathbb{B}^d$. The push-forward measure $h_*\mu_{\s^d}$ on $\mathbb{B}^d$ is given by 
\begin{equation}\label{eq:d ddd d}
\frac{\Gamma(d/2+1/2)}{\pi^{d/2+1/2}}\frac{{\rm d} x}{\sqrt{1-\|x\|^2}},
\end{equation}
and the assumptions in Theorem \ref{th:pullbacky} are satisfied. Let $\{T_{k,\ell}: \ell=0,1,2,\ldots; \;k=1,\ldots,r^d_\ell\}$ with $r^d_\ell:=\binom{\ell+d-1}{\ell}$ be orthonormal polynomials with respect to the measure \eqref{eq:d ddd d}, and each $T_{k,\ell}$ has total degree $\ell$. For $d=1$, this corresponds to Chebyshev polynomials. The case $d=2$ relates to generalized Zernike polynomials, cf.~\cite{Wunsche:2005ly}.

\begin{proposition}[The unit ball]
For $s>d/2$ and $h:\s^d \rightarrow \mathbb{B}^d$ as above, the space $H^s(\mathbb{B}^d)_h$ is reproduced by the kernel
\begin{equation}\label{eq:kernel balls et}
K^s_{\mathbb{B}^d}(x,y):=\sum_{\ell=0}^\infty (1+\ell(\ell+d-1))^{-s} \sum_{k=1}^{r^d_\ell}T_{k,\ell}(x)T_{k,\ell}(y),\quad x,y\in \mathbb{B}^d.
\end{equation}
\end{proposition}
For related results on approximation on $\mathbb{B}^d$, we refer to \cite{Petrushev:2008lq} and references therein. 
\begin{proof}
For $\ell=0,1,2,\ldots$, the eigenfunctions of the Laplace-Beltrami operator on the sphere associated to the eigenvalue $-\lambda_\ell=-\ell(\ell+d-1)$ are the \emph{spherical harmonics of order $\ell$}, given by the homogeneous harmonic polynomials in $d+1$ variables of exact total degree $\ell$ restricted to $\s^d$. Each eigenspace $E_\ell$ associated to $\lambda_\ell$ splits orthogonally into $E_\ell = E_\ell^{(1)} \oplus E_\ell^{(2)}$, where 
\begin{align}\label{eq:split eigenspaces}
\begin{split}
E_\ell^{(1)} &:= \{f\in E_\ell : f(x) = f(h(x),-x_{d+1}),\;\forall x\in\s^d\},\\
E_\ell^{(2)} &:= \{f\in E_\ell : f(x) = -f(h(x),-x_{d+1}),\;\forall x\in\s^d\},\quad \ell=0,1,2,\ldots.
\end{split}
\end{align}
We deduce from \cite[Theorem 3.3, Example 3.4]{Xu:1998rt} that the functions 
\begin{equation*}
Z^{(1)}_{k,\ell}(z) :=\|z\|^{\ell}(h^*T_{k,\ell})(\tfrac{z}{\|z\|}),\quad z\in\R^{d+1}
\end{equation*}
are homogeneous polynomials of total degree $\ell$, and their restrictions $Y^{(1)}_{k,\ell}:=Z^{(1)}_{k,\ell}|_{\s^d}$, for $k=1,\ldots,r^d_\ell$, are an orthonormal basis for $E_\ell^{(1)}$. Note that $f\in H^s(\mathbb{B}^d)_h$ if and only if $f\circ h$ is contained in 
\begin{align}
H^s(\s^d)^{\sym}&:=  \{f\in H^s(\s^d) : f|_{h^{-1}x} \text{ is constant } \forall x\in \mathbb{B}^d\}\\
& = \{f\in H^s(\s^d) : f(x)=f(h(x),-x_{d+1}), \;x\in \s^d\}. \label{eq:analog}
\end{align}
According to \eqref{eq:Bessel def} and due to the decomposition induced by \eqref{eq:split eigenspaces} and using $Y^{(1)}_{k,\ell}=h^* T_{k,\ell}$, the reproducing kernel of $H^s(\s^d)^{\sym}$ is 
\begin{align}\label{eq:K KK}
\begin{split}
K^{s}_{\s^d,\sym}(x,y) & =\sum_{\ell=0}^\infty  (1+\lambda_\ell)^{-s} \sum_{k=1}^{r^d_\ell}Y^{(1)}_{k,\ell}(x)Y^{(1)}_{k,\ell}(y)\\
& = \sum_{\ell=0}^\infty (1+\lambda_\ell)^{-s} \sum_{k=1}^{r^d_\ell}(h^*T_{k,\ell})(x) (h^*T_{k,\ell})(y),\quad x,y\in\s^d.
\end{split}
\end{align}
Thus, $H^s(\mathbb{B}^d)_h$ is indeed reproduced by \eqref{eq:kernel balls et}.
\end{proof}
\begin{example}[$\mathbb{B}^1$]
For $d=1$, the even and odd spherical harmonics, 
\begin{align}\label{eq:align 1}
\begin{split}
Y^{(1)}_{\ell}(\cos(\alpha),\sin(\alpha)) & = \sqrt{2}\cos(\ell \alpha),\\
Y^{(2)}_{\ell}(\cos(\alpha),\sin(\alpha)) & = \sqrt{2}\sin(\ell \alpha),\quad \alpha\in[0,2\pi],
\end{split}
\end{align}
with $\ell=1,2,3,\ldots$, and $Y^{(1)}_0=1$, form orthonormal bases for the respective spaces $E_\ell^{(1)}$ and $E_\ell^{(2)}$ in \eqref{eq:split eigenspaces}. We observe 
$ \dist_{[-1,1],h}(x,y):=|\arccos(x)-\arccos(y)|$, $x,y\in [-1,1]$, and recognize the Chebyshev measure in \eqref{eq:d ddd d}. The Chebyshev polynomials $T_\ell$ of the first kind, scaled by the factor $\sqrt{2}$ for $\ell=1,2,3,\ldots$, indeed satisfy the characteristic identities $T_\ell(\cos(\alpha)) = \sqrt{2}\cos(\ell \alpha)$ for $\alpha\in[0,2\pi]$, $\ell=1,2,3,\ldots$, and $T_0=1$. 

To simplify numerical experiments, we observe that the kernel 
\begin{align*}
K_5(x,y) &:= 2 - \sqrt{1-xy+|x-y|}\\
&=  2 + \frac{4}{\pi} \sum_{\ell=0}^\infty \frac{1}{4\ell^2-1} T_\ell(x)T_\ell(y)
\intertext{reproduces $H^1(\mathbb{B}^1)_h$ with an equivalent norm, and, for fixed $0<r<1$, the smooth kernel }
K_6(r;x,y) &:= \frac{(1-r^2)(1-2rxy+r^2)}{1+r^4 - 4xy(r+r^3)+ r^2(4x^2+4y^2-2)} \\
&= \frac{1}{2}+\frac{1}{2}\sum_{\ell=0}^\infty r^\ell T_\ell(x)T_\ell(y)
\end{align*}
reproduces a function space that is continuously embedded into $H^s(\mathbb{B}^1)_h$ for all $s>1/2$. As in our previous examples, our numerical experiments in Figure \ref{fig:222b} are in accordance with the theoretical results. 

\begin{figure}
\subfigure[Optimized weights for $K_5$ yield decay $-1$]{
  \includegraphics[width=.45\textwidth]{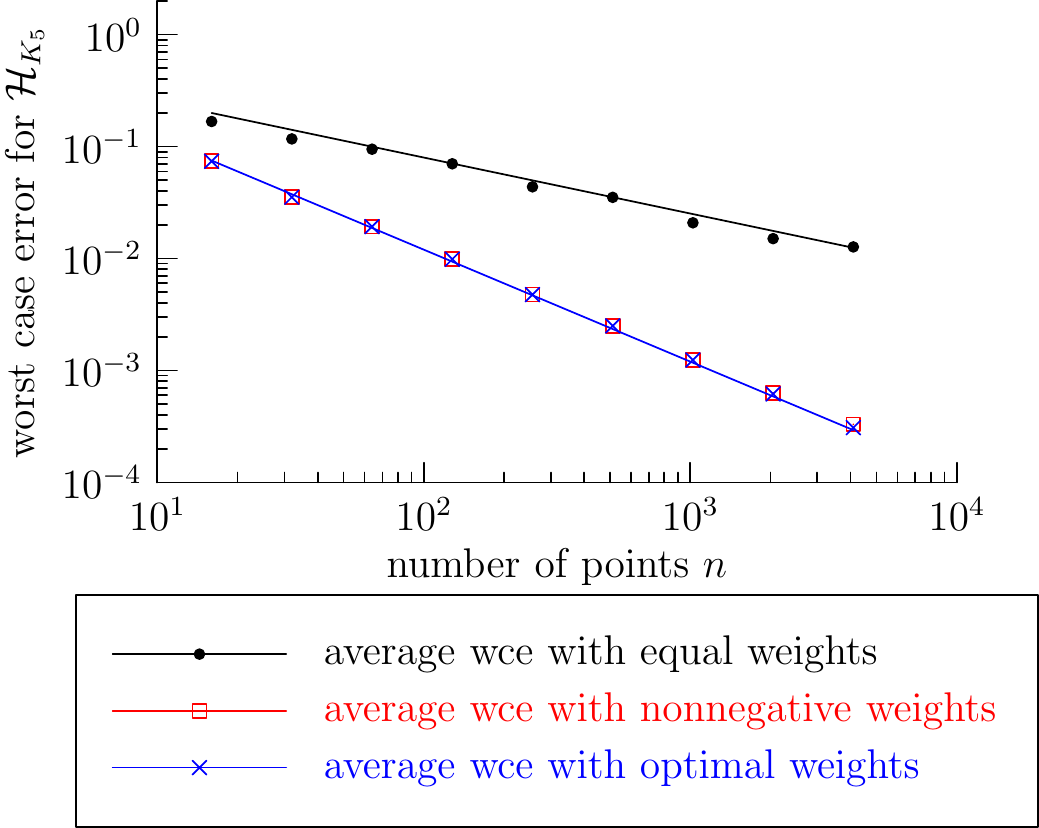}\label{subfig:3v}
 }
 \hfill
 \subfigure[Fixed weights $1/n$ for $K_6$ are stuck with decay $-1/2$]{
 \includegraphics[width=.45\textwidth]{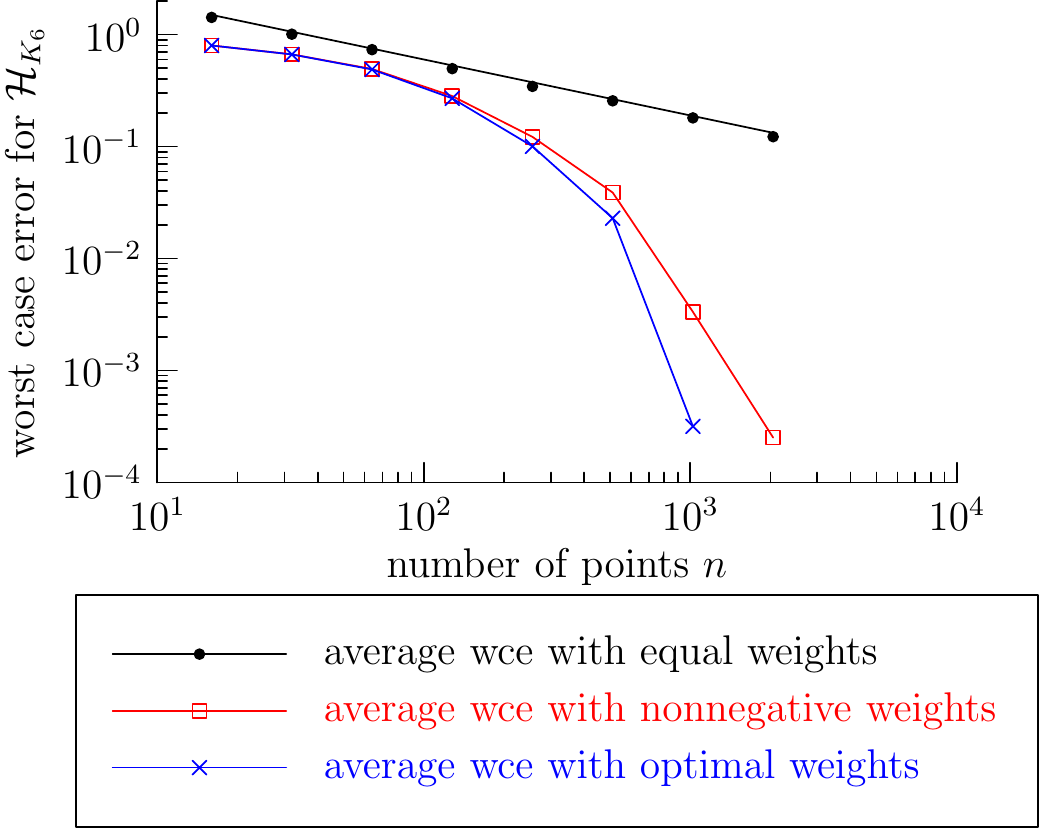}\label{subfig:4v}
}
\caption{$\wce$ for $\mathcal{H}_{K_5}$ and $\mathcal{H}_{K_6}$ with $r=0.97$ and random points in logarithmic scalings.}\label{fig:222b}
\end{figure}
\end{example}

\begin{example}[$\mathbb{B}^2$]
For $r\geq 1$, we define the family of kernels 
\begin{equation*}
L_r(x,y) := 3 - \frac{3}{2\sqrt{2r+2}}\sqrt{r-\langle x,y \rangle +\sqrt{\big(r-\langle x,y\rangle \big)^2-(1-\|x\|^2)(1-\|y\|^2)}}.
\end{equation*}
Those kernels are positive definite on $\mathbb{B}^2$ and satisfy
\begin{equation*}
\int_{\mathbb{B}^2} L_r(x,y)\frac{{\rm d}x}{2\pi \sqrt{1-\|x\|^2}}
 = \frac{3+r+2\sqrt{r^2-1}}{1+r+\sqrt{r^2-1}}.
\end{equation*}
Note that $L_1$ reproduces $H^{3/2}(\mathbb{B}^2)$ with an equivalent norm and $L_r$ for $r>1$ reproduces a space that is continuously embedded into each $H^s(\mathbb{B}^2)$ for $s>1$. 
In our numerical experiments, we set
\begin{equation*}
K_7(x,y) := L_1(x,y),\qquad K_8(x,y) := L_{51/50}(x,y),
\end{equation*}
and Figure \ref{fig:222cd} supports our theoretical results. There, however, the worst case error for nonnegative weights does not show superlinear decay for smooth functions in Figure \ref{subfig:4v}, but we speculate that this is due to a numerical artifact of the very last data point. 
\begin{figure}
\subfigure[Optimized weights for $K_7$ yield decay $-3/4$]{
  \includegraphics[width=.45\textwidth]{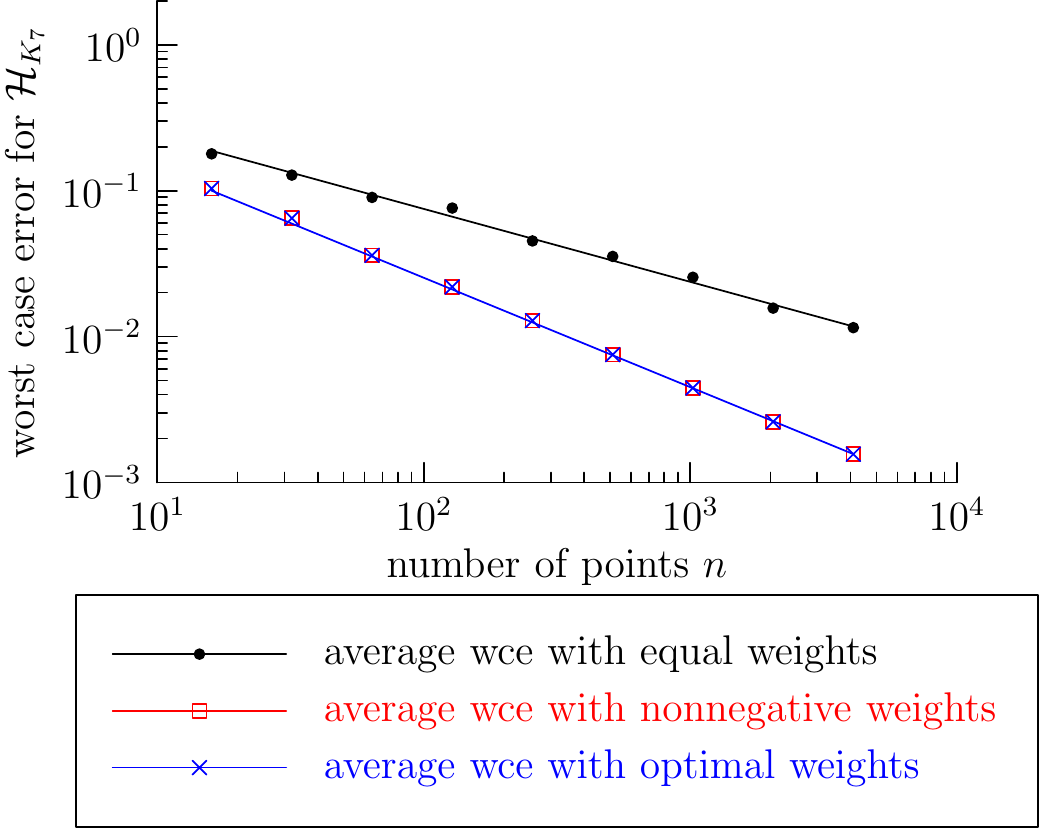}\label{subfig:3v}
 }
 \hfill
 \subfigure[Fixed weights $1/n$ for $K_8$ are stuck with decay $-1/2$]{
 \includegraphics[width=.45\textwidth]{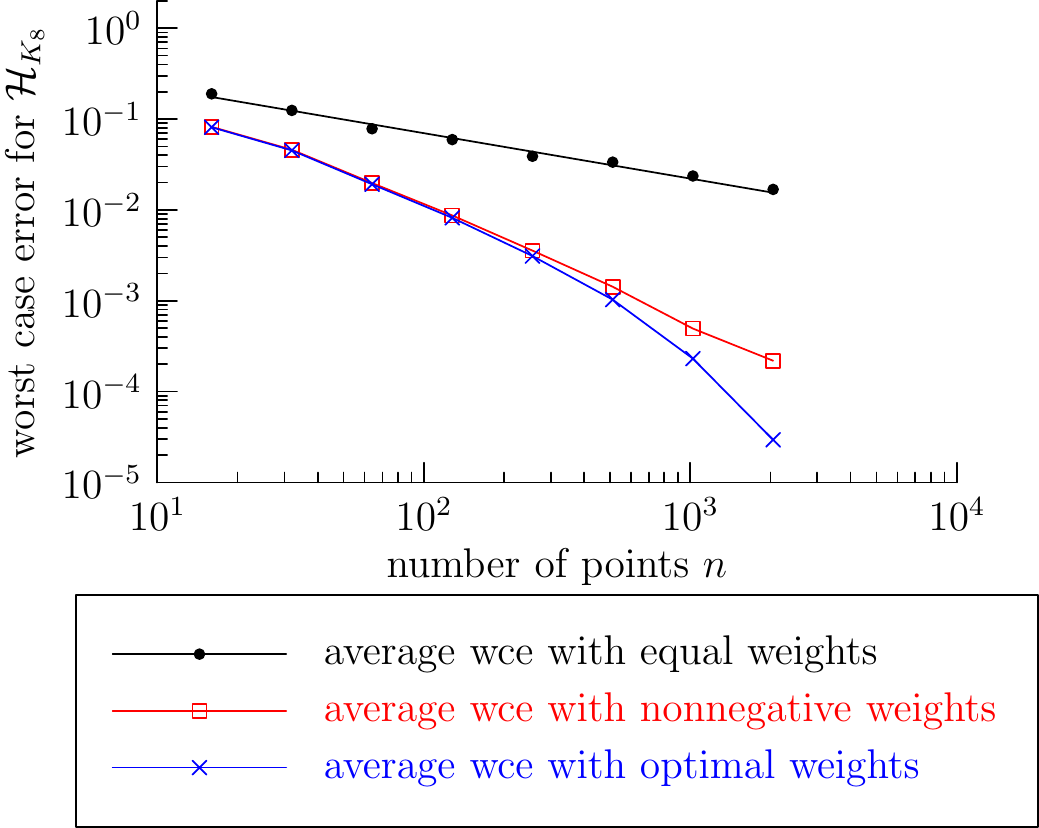}\label{subfig:4v}
}
\caption{$\wce$ for $\mathcal{H}_{K_7}$ and $\mathcal{H}_{K_8}$ for random points in logarithmic scalings.}\label{fig:222cd}
\end{figure}

\end{example}

The $d$-dimensional torus $\mathbb{T}^d:=\s^1\times \ldots\times\s^1$ leads to $h:\mathbb{T}^d \rightarrow [-1,1]^d$ defined by $h(x_1,\ldots,x_d)=\big(x_{1,1}, \ldots,x_{d,1}\big)$, where $x_i=(x_{i,1},x_{i,2})^\top\in\s^1$. The push-forward of the Riemannian measure on $\mathbb{T}^d$ under $h$ is 
\begin{equation}\label{eq:push torus}
\frac{{\rm d} x_1 \cdots {\rm d}x_d}{\pi^d\sqrt{(1-x^2_1)\cdots (1-x^2_d)}}.
\end{equation}
A suitable basis of orthonormal polynomials characterizes $H^s([-1,1]^d)_h$:
\begin{proposition}[The cube]
For $s>d/2$ and $h:\mathbb{T}^d \rightarrow [-1,1]^d$ as above, the space $H^s([-1,1]^d)_h$ is reproduced by 
\begin{equation}\label{eq:sswde 00}
K^s_{[-1,1]^d}(x,y):=\sum_{\ell\in\mathbb{N}^d} (1+\|\ell\|^2)^{-s} T_\ell(x) T_\ell(y),\quad x,y\in[-1,1]^d,
\end{equation}
where 
$ 
T_\ell(x):=T_{\ell_1}(x_1)\cdots T_{\ell_d}(x_d)$, $x\in [-1,1]^d$, $\ell\in\mathbb{N}^d$. 
\end{proposition}

\begin{proof}
The polynomials $\{T_\ell : \ell\in\mathbb{N}^d\}$ are orthonormal with respect to \eqref{eq:push torus}. 
By using the eigenspace decomposition \eqref{eq:split eigenspaces} for $\s^1$, we deduce that the space
\begin{equation*}
H^s(\mathbb{T}^d)^{\sym} := \{f\in H^s(\mathbb{T}^d) : f|_{h^{-1}x} \text{ is constant } \forall x\in [-1,1]^d\}
\end{equation*}
is reproduced by the kernel
\begin{equation*}
K^s_{\mathbb{T}^d,\sym}(x,y)=\sum_{\ell\in\mathbb{N}^d} (1+\|\ell\|^2)^{-s}  Y^{(1)}_{\ell}(x)Y^{(1)}_{\ell}(y), \quad x,y\in\mathbb{T}^d,
\end{equation*}
with $Y^{(1)}_\ell(x):=Y^{(1)}_{\ell_1}(x_1)\cdots Y^{(1)}_{\ell_d}(x_d)$ and $Y^{(1)}_{\ell_i}$ are as in \eqref{eq:align 1}. Observing $Y^{(1)}_\ell = h^* T_\ell$ concludes the proof. 
\end{proof}

By following \cite{Xu:2001fp}, we derive an analogous construction for the simplex. Define $h:\s^d\rightarrow \Sigma^d$ by $h(x):=(x_1^2,\ldots,x_d^2)$ and observe that the assumptions in Theorem \ref{th:pullbacky} are satisfied. The push-forward measure $h_*\mu_{\s^d}$ on $\Sigma^d$ is given by 
\begin{equation}\label{eq:pfwrd}
\frac{\Gamma(d/2+1/2)}{\pi^{d/2+1/2}}\frac{{\rm d} u}{\sqrt{u_1\cdots u_d(1-\sum_{i=1}^d u_i)}}.
\end{equation}
Let $\{R_{k,\ell}:\ell=0,1,2,\ldots;\; k=1,\ldots,r^d_\ell\}$ be a system of orthonormal polynomials with respect to \eqref{eq:pfwrd} on $\Sigma^d$, so that each $R_{k,\ell}$ has total degree $\ell$. 
\begin{proposition}[The simplex]
For $s>d/2$ and $h:\mathbb{T}^d \rightarrow \Sigma^d$ as above, the space $H^s(\Sigma^d)_h$ is reproduced by 
\begin{equation}\label{eq:sswde}
K^s_{\Sigma^d}(u,v):=\sum_{\ell=0}^\infty (1+2\ell(2\ell+d-1))^{-s} \sum_{k=1}^{r^d_\ell} R_{k,\ell}(u)R_{k,\ell}(v),\quad u,v\in\Sigma^d.
\end{equation}
\end{proposition}
\begin{proof}
Let us define
\begin{equation*}
Z_{k,2\ell}(z) :=\|z\|^{2\ell}R_{k,\ell}(\tfrac{z_1^2}{\|z\|^2},\ldots,\tfrac{z_d^2}{\|z\|^2}),\quad z\in\R^{d+1}.
\end{equation*}
Note that the restrictions $Y_{k,2\ell}:=Z_{k,2\ell}|_{\s^d}$ satisfy $Y_{k,2\ell}=h^* R_{k,\ell}$. We deduce from \cite{Xu:2001fp} that the collection $\{Y_{k,2\ell}:k=1,\ldots,r^d_\ell\}$ is an orthonormal system of spherical harmonics of order $2\ell$ and that the space 
\begin{equation*}
H^s(\s^d)^{\sym} := \{f\in H^s(\s^d) : f|_{h^{-1}x} \text{ is constant } \forall x\in \Sigma^d\}
\end{equation*}
is reproduced by the kernel
\begin{align*}
K(x,y)&:=\sum_{\ell=0}^\infty (1+2\ell(2\ell+d-1))^{-s} \sum_{k=1}^{r^d_\ell} Y_{k,2\ell}(x)Y_{k,2\ell}(y)\\
& = \sum_{\ell=0}^\infty (1+2\ell(2\ell+d-1))^{-s} \sum_{k=1}^{r^d_\ell} (h^*R_{k,\ell})(x) (h^*R_{k,\ell})(y),
\end{align*}
which concludes the proof.
\end{proof}

\begin{remark}
Our Theorem \ref{th:pullbacky} is an elementary way to transfer results from closed manifolds to more general setting. Our treatment of the unit ball, the cube, and the simplex were based on this transfer. The proof of the underlying Theorem \ref{th:fundamentals 2} is based on results in \cite{Filbir:2010aa}, and we restricted attention to closed manifolds although the setting in \cite{Filbir:2010aa} is more general. Alternatively, we could have stated our Theorem \ref{th:fundamentals 2} in more generality and then attempted to check that the technical requirements in \cite{Filbir:2010aa} hold. For instance, technical requirements for $[-1,1]$ were checked in \cite{Coulhon:2012sf}, and the recent work \cite{Kerkyacharian:2018ft} covers technical details for the unit ball and the simplex. 
\end{remark}

\section{Perspectives}
Re-weighting techniques for statistical and numerical integration have attracted attention in different disciplines. Partially complementing findings in \cite{Briol:2016dz,Oettershagen:2017}, 
we have here established that re-weighting random points can yield almost optimal approximation rates of the worst case integration error for isotropic Sobolev spaces on closed Riemannian manifolds. Our results suggest several directions for future work, for instance, allowing for more general spaces $\M$, considering other smoothness classes than $H^s_p(\M)$, and replacing the expected worst case error $\wce$ by alternative error functionals such as the average error, cf.~\cite{Nowak:2010rr,Ritter:2000xd}. 

\bibliographystyle{amsplain}
\bibliography{../biblio_ehler2}

\end{document}